\documentclass[11pt]{article}
\usepackage{latexsym, amssymb, amsthm, amsmath}
\usepackage{dsfont}
\usepackage{array}
\usepackage{multirow}
\usepackage{amsxtra}
\usepackage[T1]{fontenc}
\usepackage{tikz}

\DeclareMathAlphabet\gothic{U}{euf}{m}{n}

\makeatletter
\def\eqnarray{\stepcounter{equation}\let\@currentlabel=\theequation
\global\@eqnswtrue
\tabskip\@centering\let\\=\@eqncr
$$\halign to \displaywidth\bgroup\hfil\global\@eqcnt\z@
  $\displaystyle\tabskip\z@{##}$&\global\@eqcnt\@ne
  \hfil$\displaystyle{{}##{}}$\hfil
  &\global\@eqcnt\tw@ $\displaystyle{##}$\hfil
  \tabskip\@centering&\llap{##}\tabskip\z@\cr}

\def\endeqnarray{\@@eqncr\egroup
      \global\advance\c@equation\m@ne$$\global\@ignoretrue}

\def\@yeqncr{\@ifnextchar [{\@xeqncr}{\@xeqncr[5pt]}}
\makeatother

\allowdisplaybreaks[1]

\newtheorem{lemm}{Lemma}[section]
\newtheorem{thrm}[lemm]{Theorem}

\newtheorem{eeg}[lemm]{Example}
\newtheorem{rrema}[lemm]{Remark}
\newtheorem{prop}[lemm]{Proposition}
\newtheorem{ddefi}[lemm]{Definition}

\newenvironment{defi}{\begin{ddefi} \rm}{\end{ddefi}}

\newcommand{\gota}{\gothic{a}}

\newcounter{teller}

\newcounter{tellerr}

\newenvironment{tabeleq}{\begin{list}%
{\rm  (\roman{tellerr})\hfill}{\usecounter{tellerr} \leftmargin=1.1cm
\labelwidth=1.1cm \labelsep=0cm \parsep=0cm}
                         }{\end{list}}

\newcounter{tellerrr}

\newcounter{proofstep}

\newcommand{\Ni}{\mathds{N}}

\newcommand{\Ri}{\mathds{R}}
\newcommand{\Ci}{\mathds{C}}

\newcommand{\supp}{\mathrm{supp\,}}
\newcommand{\D}{\partial}

\newlength{\hightcharacter}
\newlength{\widthcharacter}

\makeatletter
\renewcommand*\env@matrix[1][*\c@MaxMatrixCols c]{%
  \hskip -\arraycolsep
  \let\@ifnextchar\new@ifnextchar
  \array{#1}}
\makeatother

\begin{document}

\thispagestyle{empty}

\vspace*{1cm}
\begin{center}
{\Large\bf Lusin characterisation of Hardy spaces associated with \\ Hermite operators} \\[5mm]
\large Tan Duc Do$^1$, Trong Ngoc Nguyen$^2$ and Truong Xuan Le$^{3,*}$ \\[10mm]

\end{center}

\vspace{5mm}

\begin{center}
{\bf Abstract}
\end{center}

\begin{list}{}{\leftmargin=1.8cm \rightmargin=1.8cm \listparindent=10mm 
   \parsep=0pt}
\item
Let $d \in \{3, 4, 5, \ldots\}$ and $p \in (0,1]$.
We consider the Hermite operator $L = -\Delta + |x|^2$ on its maximal domain in $L^2(\Ri^d)$.
Let $H_L^p(\Ri^d)$ be the completion of 
$
\{ f \in L^2(\Ri^d): \mathcal{M}_L f \in L^p(\Ri^d) \}
$
with respect to the quasi-norm
$
\|\cdot\|_{H_L^p} = \|\mathcal{M}\cdot\|_{L^p}, 
$
where $\mathcal{M}_L f(\cdot) = \sup_{t > 0} |e^{-tL} f(\cdot)|$ for all $f \in L^2(\Ri^d)$.
We characterise $H_L^p(\Ri^d)$ in terms of Lusin integrals associated with Hermite operator.
\end{list}

\vspace{2.5cm}
\noindent
January 2019

\vspace{5mm}
\noindent
AMS Subject Classification: 42B30, 42B37.

\vspace{5mm}
\noindent
Keywords: Hermite operator, Hardy space, Lusin integral, atom decomposition, heat kernel, Poisson semigroup.

\vspace{10mm}

\noindent
{\bf Home institution:}    \\[3mm]
1. \quad Vietnamese-German University  \quad Email: tan.dd@vgu.edu.vn
\\
2. \quad University of Pedagogy, Ho Chi Minh City \quad Email: trongnn37@gmail.com
\\
3. \quad University of Economic, Ho Chi Minh City \quad Email: lxuantruong@gmail.com
\\
*. \quad Corresponding author

\newpage

\setcounter{page}{1}

\section{Introduction}
Hermite operators have been studied over the years due to their significant role in harmonic oscillator.
Many of their interesting properties have been discovered. 
An in-depth study of Hermite operators can be found in the monograph \cite{Tha}.
On the other hand, Hardy spaces with their rich structures are of fundamental interest in harmonic analysis (cf.\ \cite{Ste}).
In this paper we will examine Hardy spaces with index $p \in (0,1]$ associated with Hermite operators.
We aim to characterise these spaces in terms of Lusin area integrals.
Similar results are available in the literature (cf.\ \cite[Section III.4.4]{Ste}, \cite{SY}, \cite{JPY}, etc.\ and references therein).
Our approach differs these in that the Lusin area integrals of our consideration are more general and are defined using derivatives of the form $\D_j + x_j$ which were first suggested by Thangavelu in \cite{Tha2} and are specific to Hermite operators.
Our work is motivated by \cite{Jiz} whose results are for Hardy spaces associated with Hermite operators with index $p=1$.

Next we formulate our problem.
Let $d \in \Ni$ with $d \geq 3$.
Consider the sesquilinear form
\[
\gota_0(u,v) = \int_{\Ri^d} Du \cdot Dv + |x|^2 u \, v,
\]
where $D = (\D_1, \ldots, \D_d)$, on the domain $D(\gota_0) = C_c^\infty(\Ri^d)$.
Then $\gota_0$ is positive.
Consequently it follows from \cite[Theorem VI.1.27]{Kat1} that $\gota_0$ is closable.

Let $L$ be the operator associated with the closure $\overline{\gota_0}$ in the sense of Kato's First Representation Theorem \cite[Theorem VI.2.1]{Kat1}.
Then $C_c^\infty(\Ri^d) \subset D(L)$ and
\[
Lu = -\Delta u + |x|^2 \, u
\]
for all $u \in C_c^\infty(\Ri^d)$.
In the literature $L$ is known as Hermite operator.
%From now on, all the quantities and the domains we consider in this paper are associated with the operator $L$.
%Therefore, we will not use the letter $L$ in the superscripts and subscripts of our notation for ease of writing.
It is well-known that $L$ generates a contraction $C_0$-semigroup $T$ on $L^2(\Ri^d)$.
If $f \in L^2(\Ri^d)$ then
\[
T_t f = e^{-tL} f = \sum_{n=0}^\infty e^{-t(2n+d)} \, f_n,
\]
where
\[
f_n = \sum_{|\alpha| = n} (f, h_\alpha) \, h_\alpha.
\]
The Poisson semigroup $P$ on $L^2(\Ri^d)$ associated with $L$ is given by
\[
P_t = e^{-tL^{1/2}} f = \sum_{n=0}^\infty e^{-t(2n+d)^{1/2}} \, f_n.
\]
It can be shown that $P$ is also a contraction $C_0$-semigroup on $L^2(\Ri^d)$.

Let $p \in (0, 1]$.
We define $H^p_L(\Ri^d)$ as the completion of 
\[
\{ f \in L^2(\Ri^d): \mathcal{M}_L f \in L^p(\Ri^d) \}
\]
under the quasi-norm
\[
\|\cdot\|_{H_L^p} = \|\mathcal{M_L}\cdot\|_{L^p}, 
\]
where $\mathcal{M}_L f(\cdot) = \sup_{t > 0} |T_t f(\cdot)|$ for all $f \in L^2(\Ri^d)$.

For each $x \in \Ri^d$, we define the Lusin integral associated with Hermite operator as
\begin{equation} \label{Sop}
Sf(x) 
:= \left( \int_{\Gamma(x)} t^{1-d} \, |\nabla_L P_t f(y)|^2 \, dy \, dt \right)^{1/2}
\end{equation}
where $\nabla_L = (\D_t, \D_1 + x_1, \ldots, \D_d + x_d)$ and
\begin{equation} \label{Gamma1}
\Gamma(x) := \{ (y,t) \in \Ri^d \times (0,\infty): |x-y| < t \}
\end{equation}
Let $H_S^p(\Ri^d)$ be the completion of
\[
\big\{f \in L^2(\Ri^d): Sf \in L^p(\Ri^d) \big\}
\]
under the quasi-norm
\[
\|\cdot\|_{H_S^p(\Ri^d)} = \|S\cdot\|_{L^p(\Ri^d)}.
\]

In this paper we aim to characterise the space $H^p_L(\Ri^d)$ in terms of the Lusin integral defined above.
Our main result is the following.

\begin{thrm} \label{main theorem}
We have
\[
\left( H^p_L(\Ri^d), \|\cdot\|_{H^p_L(\Ri^d)} \right) 
= \left( H_S^p(\Ri^d), \|\cdot\|_{H_{S}^p(\Ri^d)} \right).
\]
\end{thrm}

Following \cite{Tha2}, we also consider the Riesz transforms
\[
R_j^L := \left( \D_j + x_j \right) L^{-1/2}, \quad j \in \{1,2,3,\ldots\}.
\]
We have the following.

\begin{thrm} \label{Riesz bounded}
The operators $R_j^L$ is bounded on $H^p_L(\Ri^d)$ for all $j \in \{1,2,3,\ldots\}$.
\end{thrm}

%Theorems \ref{main theorem} and \ref{Riesz bounded} together imply the following characterisation of the space $H^p_L(\Ri^d)$ in terms of the Riesz transforms $R_j^L$'s.
%
%\begin{coro} \label{Riesz char}
%The space $H^p_L(\Ri^d)$ is the completion of 
%\[
%\{ f \in L^2(\Ri^d): R_j^L f \in L^p(\Ri^d) \}
%\]
%under the quasi-norm
%\[
%\|\cdot\|_{R} := \|\cdot\|_{L^p(\Ri^d)} + \sum_{j=1}^d \|R_j^L \cdot\|_{L^p(\Ri^d)}.
%\] 
%\end{coro}

The outline of the paper is as follows.
In proving Theorem \ref{main theorem} we will need two intermediate spaces $H_{L,\max}(\Ri^d)$ and $H_A^p(\Ri^d)$ which will be defined in the corresponding sections.
In Section \ref{pre} we provide some preliminaries for later use.
In Section \ref{max} we show that $H_L^p(\Ri^d) \subset H_{L,\max}(\Ri^d)$.
In Section \ref{A} we give a characterisation $H_L^p(\Ri^d) = H_A^p(\Ri^d) \cap L^2(\Ri^d)$.
In Section \ref{S} we prove the main theorem.
In Section \ref{Riesz} we prove Theorem \ref{Riesz bounded}.

Throughout the paper, we let $C$ be a positive constant independent of the main parameters whose value varies from line to line.
%We will write $A \lesssim B$ if there is a constant $C$ such that $A \leq C B$ and $A \sim B$ if $A \lesssim B$ and $B \lesssim A$.
We also set $\Ni = \{0, 1, 2, \ldots\}$.

\section{Some prior estimates} \label{pre}

Let $\rho(x) = \displaystyle\frac{1}{1+|x|}$.

\begin{prop}[{\cite[Lemma 1.4]{She}}] \label{rho}
There exist $M>0$ and $k>0$ such that
\[
\frac{1}{M} \left( 1 + \frac{|x-y|}{\rho(x)} \right)^{-k} 
\leq \frac{\rho(y)}{\rho(x)} 
\leq M \left(1 + \frac{|x-y|}{\rho(x)} \right)^{-\frac{k}{k+1}}.
\]
In particular, $\rho(x) \sim \rho(y)$ if $|x-y| \lesssim \rho(x)$.
\end{prop}

Next we consider some kernel estimates of the $C_0$-semigroup $T$ generated by $L$.

\begin{prop}
Let $g_t$ be the heat kernel of $T_t$ for each $t > 0$.
Then following hold.

\begin{tabeleq}
\item There exists a $C > 0$ such that for every $N>0$, there is a constant $C_N > 0$ that satisfies 
\[
0 \leq g_t(x,y) \leq C_N \, t^{-d/2} \, e^{-C|x-y|^2/t} \left( 1 + \frac{\sqrt{t}}{\rho(x)} \right)^{-N} \left( 1 + \frac{\sqrt{t}}{\rho(y)} \right)^{-N} 
\]
for all $x, y \in \Ri^d$.
\item There exist a $\delta \in (0,1)$ and a $C > 0$ such that for every $N>0$, there is a constant $C_N$ such that
\begin{equation} \label{kernel lipschitz}
|g_t(x+h,y) - g_t(x,y)| \leq C_N \left( \frac{|h|}{\sqrt{t}} \right)^\delta t^{-d/2} \, e^{-C|x-y|^2/t} \left( 1 + \frac{\sqrt{t}}{\rho(x)} \right)^{-N} \left( 1 + \frac{\sqrt{t}}{\rho(y)} \right)^{-N} 
\end{equation}
for all $x, y \in \Ri^d$ and $|h| < \sqrt{t}$.
\end{tabeleq}

\end{prop}

\begin{proof}
This follows from \cite[Lemma 2.2 and Proposition 3.1]{JPY}.
\end{proof}

\begin{prop} \label{Poisson kernel}
Let $p_t$ be the heat kernel of $P_t$ for each $t > 0$.
Then following hold.
\begin{tabeleq}
\item For every $N>0$, there is a constant $C_N > 0$ such that 
\[
0 \leq p_t(x,y) \leq C_N \, \frac{t}{(t + 2|x-y|)^{d+1}} \left( 1 + \frac{\sqrt{t}}{\rho(x)} \right)^{-N} \left( 1 + \frac{\sqrt{t}}{\rho(y)} \right)^{-N} 
\]
for all $x, y \in \Ri^d$.
\item There exist a $\delta \in (0,1)$ and a $C > 0$ such that for every $N>0$, there is a constant $C_N$ such that
\begin{equation} \label{poisson kernel lipschitz}
|p_t(x+h,y) - p_t(x,y)| \leq C_N \left( \frac{|h|}{t} \right)^\delta \frac{t}{(t + 2|x-y|)^{d+1}} \, \left( 1 + \frac{\sqrt{t}}{\rho(x)} \right)^{-N} \left( 1 + \frac{\sqrt{t}}{\rho(y)} \right)^{-N} 
\end{equation}
for all $x, y \in \Ri^d$ and $|h| < \sqrt{t}$.
\end{tabeleq}
\end{prop}

\begin{proof}
This follows from \cite[Propositions 3.4 and 3.5]{JPY}.
\end{proof}

%A Cacciopoli-type inequality
%
%
%\begin{lemm}[{\cite[Lemma 2]{Jiz}}]
%Let $r > 0$, $(x_0,t_0) \in \Ri^d \times (0,\infty)$ and $\Omega = B((x_0,t_0),r)$.
%Assume $u \in L^2_\loc(4\Omega)$ is a weak solution of $-\Delta u + \D_t^2 u + |x|^2 u = 0$ in $4\Omega$.
%Then there exists a $C > 0$ such that
%\[
%\int_\Omega \left( |\D_t u(x,t)|^2 + \sum_{i=1}^d |\delta_i u(x,t)|^2 \right) dx \, dt 
%\leq \frac{C}{r^2} \int_{2\Omega} |u(x,t)|^2 \, dx \, dt.
%\]
%\end{lemm}

We end this preliminary section with an atom decomposition of the space $H_L^p(\Ri^d)$.
Let $p \in (0, 1]$.
We first define $H_L^p$-atoms.
\begin{defi}
Let $x_0 \in \Ri^d$ and $r > 0$.
A function $a$ is an $H_L^p$-atom associated with a ball $B(x_0,r)$ if
\begin{tabeleq}
\item $\supp a \subset B(x_0, r)$,
\item $\|a\|_{L^\infty} \leq |B(x_0,r)|^{-1/p}$ and
\item $\int_{\Ri^d} a = 0$ if $r < \rho(x_0)/4$ (moment condition).
\end{tabeleq}
\end{defi}

It is of fundamental importance that each $H_L^p$-function can be written as the sum of $H_L^p$-atoms and that the $H_L^p$-quasi-norm is equivalent to that given by the atom decomposition.

\begin{prop} \label{atom decom}
Let $\frac{d}{d+\delta} < p \leq 1$, where $\delta$ is given by \eqref{kernel lipschitz}.
Then for all $f \in H_L^p(\Ri^d)$, there exist a sequence $\{c_j\} \subset l^p(\Ri)$ and $H_L^p$-atoms $a_j$'s such that
\begin{equation} \label{f into atoms}
f = \sum_{j \in \Ni} c_j a_j
\end{equation}
in $H_L^p(\Ri^d)$.

For all $f \in H_L^p(\Ri^d)$, define the quasi-norm
\[
\|f\|_{L,at} = \inf\left\{ \Big( \sum_{j \in \Ni} |c_j|^p \Big)^{1/p} \right\}
\]
where the infimum is taken over all decompositions \eqref{f into atoms}.
Then
\[
\left( H_L^p(\Ri^d), \|\cdot\|_{H_L^p} \right) = \left( H_L^p(\Ri^d), \|\cdot\|_{L,at} \right).
\]
\end{prop}

\begin{proof}
This follows from \cite[Theorem 2.15]{BDL}.
\end{proof}

\section{$H_L^p \subset H_{L,\max}^p$} \label{max}

For each $x \in \Ri^d$ and $f \in L^2(\Ri^d)$, define
\[
f_L^*(x) = \sup_{\{(t,y) \in \Ri_+ \times \Ri^d: |x-y|<t\}} |P_t f(y)|.
\]
Let $p \in (0,1]$.
We define $H_{L,\max}^p(\Ri^d)$ be the completion of
\[
\{f \in L^2(\Ri^d): f_L^* \in L_p(\Ri^d)\}
\]
under the quasi-norm
\[
\|f\|_{H_{L,\max}^p(\Ri^d)} = \|f_L^*\|_{L^p(\Ri^d)}, \quad f \in H_{L,\max}^p(\Ri^d).
\]

\begin{prop} \label{HpL subset HpLmap}
Let $\frac{d}{d+\delta} < p \leq 1$, where $\delta$ is given by \eqref{poisson kernel lipschitz}.
Then
\[
H_L^p(\Ri^d) \subset H_{L,\max}^p(\Ri^d).
\] 
\end{prop}

\begin{proof}
Due to Proposition \ref{atom decom}, it suffices to check that there exists a $C>0$ such that
$
\|a\|_{H_{L,\max}^p} \leq C
$
for all $H_L^p$-atom $a$.

Let $a$ be an $H_L^p$-atom.
Let $y_0 \in \Ri^d$ and $r > 0$ such that $\supp a \subset B(y_0, r) =: B$.
Consider 
\[
\|a\|_{H_{L,\max}^p} 
= \int_{\Ri^d} |a^*_L|^p 
= \int_{2B} |a^*_L|^p + \int_{(2B)^C} |a^*_L|^p
=: I_1 + I_2,
\]
where $2B := B(y_0,2r)$ and $(2B)^C = \Ri^d \setminus (2B)$.

For $I_1$, we have
\[
|P_t a(y)| 
\leq \|a\|_{L^\infty} \int_{B} |p_t(y,z)| \, dz 
\leq C \, \|a\|_{L^\infty} \int_{B} \frac{t}{(t + |y-z|)^{d+1}} \, dz
\leq C \, |2B|^{-1/p}
\]
for all $y \in B$, where we used Proposition \ref{Poisson kernel} in the second step and \cite[Lemma 2.1]{BDY} in the third step.
It follows that
$
I_1 \leq C.
$

To estimate $I_2$, we first show that there exists a $C > 0$ such that
\begin{equation} \label{2BC}
|a_L^*(x)| \leq C \, |B|^{1+1/d-1/p} \, \frac{1}{|x-y_0|^{d+1}}
\end{equation}
for all $x \in (2B)^C$.

Let $x \in (2B)^C$.
We consider the following two cases.

\noindent
{\bf Case 1:} Suppose that $r < \rho(y_0)/4$. 
Then $a$ satisfies the moment condition.
For all $y \in \Ri^d$ and $t > 0$ such that $|x-y| < t$, we have
\begin{eqnarray*}
|P_t a(y)| &=& \Big| \int_{B} \left( p_t(y,z) - p_t(y,y_0) \right) \, a(z) \, dz \, \Big|
\leq \|a\|_{L^\infty} \int_{B} \left| p_t(y,z) - p_t(y,y_0) \right| \, dz
\\
&\leq& C \, \|a\|_{L^\infty} \int_{B} \left( \frac{|z-y_0|}{t} \right)^\delta \frac{t}{(t+|y-y_0|)^{d+1}} \, \left( \frac{\rho(y_0)}{t} \right)^{1-\delta} \, dz
\\
&\leq& C \, \|a\|_{L^\infty} \int_{B} \left( \frac{|z-y_0|}{t} \right)^\delta \frac{t}{(|x-y|+|y-y_0|)^{d+1}} \, \left( \frac{\rho(y_0)}{t} \right)^{1-\delta} \, dz
\\
&\leq& C \, \|a\|_{L^\infty} \int_{B} \frac{|z-y_0|^\delta \, \rho(y_0)^{1-\delta}}{|x-y_0|^{d+1}} \, dz,
\end{eqnarray*}
where we used Proposition \ref{Poisson kernel} in the third step.
Note that $|z-y_0| \leq r$ as $z \in B$.
Also $\rho(y_0) \leq Cr$ for some $C > 0$ by Proposition \ref{rho}.
These imply
\[
|P_t a(y)| 
\leq C \, \|a\|_{L^\infty} \int_{B} \frac{r}{|x-y_0|^{d+1}} \, dz
\leq C \, |B|^{1+1/d-1/p} \, \frac{1}{|x-y_0|^{d+1}}
\]
for all $y \in \Ri^d$ and $t > 0$ such that $|x-y| < t$.
Hence \eqref{2BC} follows.

\noindent
{\bf Case 2:} Suppose that $r \geq \rho(y_0)/4$.
So $a$ need not satisfy the moment condition. 
For all $y$ and $z$ such that $|x-y| < t$ and $|y_0-z|<r$, we have
\[
t + |y-z| \geq t + |x-y_0| - |x-y| - |y_0-z| \geq |x-y_0| - r \geq \frac{|x-y_0|}{2},
\]
where we used the fact that $x \in (2B)^C$ in the last step.
By Proposition \ref{rho}, there exists a $C > 0$ such that $\rho(z) \leq C r$ for all $z \in B$.
Therefore
\begin{eqnarray*}
|P_t a(y)| &=& \Big| \int_{B} p_t(y,z) \, a(z) \, dz \, \Big|
\leq \|a\|_{L^\infty} \int_{B} \left| p_t(y,z) \right| \, dz
\leq C \, \|a\|_{L^\infty} \int_{B} \frac{t}{(t+|y-z|)^{d+1}} \, \frac{\rho(z)}{t} \, dz
\\
&\leq& C \, \|a\|_{L^\infty} \int_{B} \frac{r}{|x-y_0|^{d+1}} \, dz
\leq C \, |B|^{1+1/d-1/p} \, \frac{1}{|x-y_0|^{d+1}}
\end{eqnarray*}
for all $y \in \Ri^d$ and $t > 0$ such that $|x-y| < t$, where we used Proposition \ref{Poisson kernel} in the third step.
Hence \eqref{2BC} follows.

Having proved \eqref{2BC}, we now obtain
\[
I_2 
= \int_{(2B)^C} |a_L^*|^p 
\leq C |B|^{(1+1/d-1/p)p} \int_{(2B)^C} \frac{1}{|x-y_0|^{(d+1)p}} \, dx
\leq C |B|^{p+p/d-1} \, |B|^{1-(d+1)p/d} = C.
\]
Thus the proposition follows.
\end{proof}

\section{$H_L^p = H_A^p \cap L^2$} \label{A}

The work in this section is inspired by \cite[Chapter 4]{HLMMY} whose results are for $p=1$ on spaces of homogeneous type.

Define 
\[
Af(x) 
:= \left( \int_0^\infty \int_{|x-y|<t} \left|t \, (\D_tP_t f)(y) \right|^2 \, \frac{dy \, dt}{t^{d+1}} \right)^{1/2}
= \left( \int_0^\infty \int_{|x-y|<t} t^{1-d} \, \left|(\D_tP_t f)(y) \right|^2 \, dy \, dt \right)^{1/2},
\]
where $f \in L^2(\Ri^d)$ and $x \in \Ri^d$.

Next we show that $A$ is a bounded operator on $L^2(\Ri^d)$.
For this we need the following lemma.

\begin{lemm} \label{aux lemm}
Let $s > 0$.
Let $\psi: (0,\infty) \longrightarrow \Ci$ be such that 
\[
|\psi(t)| \leq C \frac{t^s}{1 + t^{2s}}
\]
for some $C > 0$ and for all $t > 0$.
Then there exists a $C > 0$ such that
\[
\left( \int_0^\infty \|\psi(t\sqrt{L})\|_{L^2(\Ri^d)}^2 \, \frac{dt}{t} \right)^{1/2}
= C \, \|f\|_{L^2(\Ri^d)}.
\]
\end{lemm}

\begin{proof}
Let $C := (\int_0^\infty |\psi(t)|^2 \frac{dt}{t})^{1/2} < \infty$.
Then
\begin{eqnarray*}
\int_0^\infty \|\psi(t\sqrt{L})\|_{L^2(\Ri^d)}^2 \, \frac{dt}{t}
&=& \int_0^\infty \Big( \psi(t\sqrt{L}) f, \psi(t\sqrt{L}) f \Big) \, \frac{dt}{t}
= \left( \int_0^\infty |\psi|^2(t\sqrt{L}) \, \frac{dt}{t} f, f \right)
\\
&=& \left( \int_0^\infty \Big( \int_0^\infty |\psi|^2(t\sqrt{\lambda}) \, dE_{\sqrt{L}}(\lambda) \Big) \, \frac{dt}{t} f, f \right)
\\
&=& \int_0^\infty \left( \int_0^\infty |\psi|^2(t\sqrt{\lambda}) \, \frac{dt}{t} f, f \right) \, dE_{\sqrt{L}}(\lambda)
= C \|f\|_{L^2(\Ri^d)}^2,
\end{eqnarray*}
where $E_{\sqrt{L}}(\lambda)$ is the spectral decomposition of $\sqrt{L}$.
\end{proof}

\begin{prop} \label{A L2 bounded}
The operator $A$ is bounded on $L^2(\Ri^d)$.
\end{prop}

\begin{proof}
Let $f \in L^2(\Ri^d)$. 
Then
\begin{eqnarray*}
\|Af\|_{L^2(\Ri^d)}^2
&=& \int_{\Ri^d} \int_0^\infty \int_{|x-y|<t} t^{1-d} \, \left|(\D_tP_t f)(y) \right|^2 \, dy \, dt \, dx
\\
&=& \int_0^\infty \int_{\Ri^d} \int_{|x-y|<t} t^{1-d} \, \left|(\D_tP_t f)(y) \right|^2 \, dx \, dy \, dt
\\
&=& C \int_0^\infty \int_{\Ri^d} t \, \left|(\D_tP_t f)(y) \right|^2 \, dy \, dt
= C \int_{\Ri^d} \int_0^\infty \left| \big( t \, \sqrt{L} \, e^{-t\sqrt{L}} f \big)(y) \right|^2 \, \frac{dt}{t} \, dy
\\
&=& C \, \|f\|_{L^2(\Ri^d)}^2
\end{eqnarray*}
for some $C > 0$, where we used Lemma \ref{aux lemm} in the last step.
\end{proof}

Let $p \in (0,1]$.
We define $H_A^p(\Ri^d)$ as the completion of 
\[
\big\{f \in L^2(\Ri^d): Af \in L^p(\Ri^d) \big\}
\]
under the quasi-norm
\[
\|\cdot\|_{H_A^p(\Ri^d)} = \|A\cdot\|_{L^p(\Ri^d)}.
\]

\begin{defi}
Let $M$ be a positive integer.
A function $a \in L^2(\Ri^d)$ is called a $(p,2,M)$-atom associated to the operator $L$ if there exist a function $b \in D(L^M)$ and a ball $B$ such that
\begin{tabeleq}
\item $a = L^M b$,
\item $\supp L^k b \subset B$, where $k = 0, 1, \ldots, M$,
\item $\|(r^2_B L)^k b\|_{L^2(\Ri^d)} \leq r^{2M}_B \, |B|^{1/2-1/p}$, where $k = 0, 1, \ldots, M$.
\end{tabeleq}
\end{defi}

\begin{defi}
Let $f \in L^2(\Ri^d)$. 
If there exists a sequence $\{\lambda_j\} \in l^p$ such that 
\begin{equation} \label{atom decom 2}
f = \sum_{j \in \Ni} \lambda_j \, a_j
\end{equation}
in $L_2(\Ri^d)$, where each $a_j$ is a $(p,2,M)$-atom, then we say that \eqref{atom decom 2} is an atomic $(p,2,M)$-representation of $f$.
\end{defi}

Let $H_{L,at,M}^p(\Ri^d)$ be the completion of 
\[
\{f \in L_2(\Ri^d): f \mbox{ has an atomic } (p,2,M)-\mbox{representation} \}
\]
with respect to the quasi-norm 
\[
\|f\|_{H_{L,at,M}^p(\Ri^d)} := 
\inf \left\{ \left( \sum_{j \in \Ni} |\lambda_j|^p \right)^{1/p}: f = \sum_{j \in \Ni} \lambda_j \, a_j \mbox{ is an atomic } (p,2,M)-\mbox{representation} \right\}.
\]

It turns out that certain functions in $H_A^p(\Ri^d)$ can be decomposed into $(p,2,M)$-atoms.
Specifically, we will prove that $H_A^p(\Ri^d) \cap L^2(\Ri^d)$ and $H_{L,at,M}^p(\Ri^d)$ equal as quasi-norm spaces.

\begin{lemm} \label{Hat in HpA}
Let $M > \frac{d}{2}(\frac{1}{p} - \frac{1}{2})$.
Then
\[
\left( H_{L,at,M}^p(\Ri^d), \|\cdot\|_{H_{L,at,M}^1(\Ri^d)} \right)
\subset \left( H_A^p(\Ri^d) \cap L^2(\Ri^d), \|\cdot\|_{H_A^p(\Ri^d)} \right).
\]
\end{lemm}

\begin{proof}
Let $f \in H_{L,at,M}^p(\Ri^d)$.
We need $Af \in L^p(\Ri^d)$.
But $f = \sum_{j \in \Ni} \lambda_j \, a_j$, where $\{\lambda_j\} \in l^p$ and each $a_j$ is a $(p,2,M)$-atom.
Therefore it suffices to show that there exists a $C > 0$ such that $\|Aa\|_{L^p(\Ri^d)} \leq C$ for all $(p,2,M)$-atom $a$.

Let $a$ be a $(p,2,M)$-atom.
Let $x_0 \in \Ri^d$ and $r > 0$ be such that $\supp a \subset B(x_0,r) =: B$.
By a generalized Holder's inequality (cf.\ \cite[Corollary 2.5]{AF}), we have
\begin{equation} \label{Aa estimate}
\|Aa\|_{L^p(\Ri^d)} \leq C \, \sum_{j \in \Ni} |2^{j} B|^{1/q} \, \|Aa\|_{L^2(U_j)}
\end{equation}
for some $C > 0$, where $1/q = 1/p - 1/2$, $U_0 = B$ and $U_j = 2^{j} B \setminus 2^{j-1} B$ for $j \in \{1, 2, 3, \ldots\}$.
Since $A$ is bounded on $L^2(\Ri^d)$ by Proposition \ref{A L2 bounded}, we deduce that
\begin{equation} \label{Aa estimate 2}
\|Aa\|_{L^2(U_j)} \leq C \, \|a\|_{L^2(B)} \leq C \, |B|^{1/2-1/p}
\end{equation}
for some $C > 0$ and for all $j = 0,1,2$, where the last step follows from the bounded property given in the definition of a $(p,2,M)$-atom.

Let $j \geq 3$ and $b \in D(L^M)$ be such that $a = L^M b$.
Then
\begin{eqnarray}
\|A a\|^2_{L^2(U_j)}
&=& \int_{U_j} \int_0^\infty \int_{|x-y|<t} \left|t \, \big(\sqrt{L} \, P_t a \big)(y) \right|^2 \, \frac{dy \, dt}{t^{d+1}} \, dx
\nonumber
\\
&=& \int_{U_j} \int_0^\infty \int_{|x-y|<t} \left|\Big( (t\sqrt{L})^{1+2M} \, P_t b \Big)(y) \right|^2 \, \frac{dy \, dt}{t^{d+1+4M}} \, dx
\nonumber
\\
&=& \int_{U_j} \Big( \int_0^{|x-x_0|/4} + \int_{|x-x_0|/4}^\infty \Big) \int_{|x-y|<t} \left|\Big( (t\sqrt{L})^{1+2M} \, P_t b \Big)(y) \right|^2 \, \frac{dy \, dt}{t^{d+1+4M}} \, dx
\nonumber
\\
&=:& (I) + (II).
\label{Aa I&II}
\end{eqnarray}

We now estimate each term separately.
For (I), set
\[
F_j = \{ y \in \Ri^d: |x-y| \leq \frac{|x-x_0|}{4} \mbox{ for some } x \in U_j \}.
\]
If $z \in B$, $y \in F_j$ and $x \in U_j$ is such that $|x-y| \leq |x-x_0|/4$, then
\[
|y-z| \geq |x-x_0| - |x-y| - |z-x_0| \geq \frac{3}{4}|x-x_0| - r \geq \frac{|x-x_0|}{2} \geq 2^{j-2}r.
\]
We deduce that $d(F_j,B) \geq 2^{j-2}r$.
Therefore
\begin{eqnarray}
(I) 
&\leq& C \int_0^{2^{j-2}r} \int_{F_j} \left| \Big( (t\sqrt{L})^{1+2M} \, P_t b \Big)(y) \right|^2 \, \frac{dy \, dt}{t^{1+4M}}
= C \int_0^{2^{j-2}r} \left\| (t\sqrt{L})^{1+2M} \, P_t b \right\|^2_{L^2(F_j)} \, \frac{dt}{t^{1+4M}}
\nonumber
\\
&\leq& C \, \|b\|^2_{L^2(B)} \int_0^{2^{j-2}r} \Big( \frac{t}{d(F_j,B)} \Big)^{4M+2} \, \frac{dt}{t^{1+4M}}
\leq C r^{4M} |B|^{1-2/p} \int_0^{2^{j-2}r} \Big( \frac{t}{2^j r} \Big)^{4M+2} \, \frac{dt}{t^{1+4M}}
\nonumber
\\
&=& C \, |B|^{1-2/p} \, 2^{-4Mj}
= C \, |2^j B|^{1-2/p} \, 2^{-j(d(1-2/p) + 4M)},
\label{I}
\end{eqnarray}
where we used \cite[Proposition 3.1]{HLMMY} in the third step.
For (II), we have
\begin{eqnarray}
(II) 
&\leq& C \int_{2^{j-3}r}^\infty \int_{\Ri^d} \left| \Big( (t\sqrt{L})^{1+2M} \, P_t b \Big)(y) \right|^2 \, \frac{dy \, dt}{t^{1+4M}}
\leq C \, \|b\|^2_{L^2(B)} \int_{2^{j-3}r}^\infty \frac{dt}{t^{1+4M}}
\nonumber
\\
&\leq& C r^{4M} |B|^{1-2/p} \int_{2^{j-3}r}^\infty \frac{dt}{t^{1+4M}}
\leq C \, |B|^{1-2/p} \, 2^{-4Mj}
= C \, |2^j B|^{1-2/p} \, 2^{-j(d(1-2/p) + 4M)}.
\label{II}
\end{eqnarray}
It follows from \eqref{Aa estimate}, \eqref{Aa estimate 2}, \eqref{Aa I&II}, \eqref{I} and \eqref{II} that $\|Aa\|_{L^p(\Ri^d)} \leq C$.
Hence the claim follows.
\end{proof}

Next we will show that the reverse inclusion $H_{L,at,M}^p(\Ri^d) \supset H_A^p(\Ri^d) \cap L^2(\Ri^d)$ holds.
This requires some techniques from tent spaces.
Therefore we will diverge a little to study tent spaces.
The aim is to make use of the atomic decomposition already available in tent spaces (cf.\ \cite{Rus}) to study the space $H_A^p(\Ri^d) \cap L^2(\Ri^d)$ in terms of $(p,2,M)$-atoms defined above.

First we define tent spaces.
Let $\alpha > 0$.
For any closed subset $F \subset \Ri^d$, let 
\[
\mathsf{R}_\alpha(F) := \bigcup_{x \in F} \Gamma_\alpha(x),
\]
where 
\begin{equation} \label{Gamma_al}
\Gamma_\alpha(x) := \{ (y,t) \in \Ri^d \times (0,\infty): |x-y| < \alpha t \}, \quad x \in \Ri^d
\end{equation}
Note that $\Gamma_1$ agrees with $\Gamma$ defined in \eqref{Gamma1}.
If $O \subset \Ri^d$ is open then we define
\[
T_\alpha(O) 
:= (\mathsf{R}_\alpha(O^C))^C 
= \{(x,t) \in \Ri^d \times (0, \infty): d(x,O^C) \geq \alpha t\},
\]
which is called the tent over $O$ with aperture $\alpha$.
For short, we will write $\mathsf{R}(F)$ and $T(O)$ in place of $\mathsf{R}_1(O)$ and $T_1(O)$ respectively.

For each measurable function $f$ on $\Ri^d \times (0,\infty)$ and $x \in \Ri^d$, define
\[
(\mathsf{S}f)(x) = \left( \int_{\Gamma(x)\times(0,\infty)} |f(y,t)|^2 \, dy \, \frac{dt}{t^{d+1}} \right)^{1/2}
\]
We say that $f \in T^p_2(\Ri^d)$ if 
\[
\|f\|_{T^p_2(\Ri^d)} := \|\mathsf{S}f\|_{L^p(\Ri^d)} < \infty.
\]

\begin{defi}
A measurable function a on $\Ri^d × (0, \infty)$ is said to be a $T^p_2(\Ri^d)$-atom if there exists a ball $B \subset \Ri^d$ such that $a$ is supported in $T(B)$ and 
\[
\int_{\Ri^d \times (0,\infty)} |a(y,t)|^2 \, dy \, \frac{dt}{t}
\leq |B|^{1-2/p}.
\]
\end{defi}

%For each $t > 0$, let $K_{\cos(t\sqrt{L})}(\cdot,\cdot)$ be the integral kernel of the operator $\cos(t\sqrt{L})$.
%It is well-known that there is a constant $c_0 > 0$ such that
%\begin{equation} \label{cos kernel}
%\supp K_{\cos(t\sqrt{L})} \subset \{ (x,y) \in \Ri^d\times\Ri^d: |x-y| \leq c_0t \}.
%\end{equation}
In what follows, we let $\phi \in C_c^\infty(\Ri)$ be such that
\begin{tabeleq}
\item $\supp\phi \subset (-1,1)$ %, where $c_0$ is the constant in \eqref{cos kernel},
\item $\phi$ is even,
\item $\phi \geq 0$ on $(-1,1)$ and $\phi > 0$ on $(-1/2,1/2)$.
\end{tabeleq}
For each $M \geq 1$, we set $\Psi(x) := x^{2(M+1)} \, \Phi(x)$, where $\Phi$ is the Fourier transform of $\phi$ and $x \in \Ri$.
Consider the operator $\pi_{\Psi,L}: T^2_2(\Ri^d) \longrightarrow L^2(\Ri^d)$ given by
\[
\pi_{\Psi,L}(F) := \int_0^\infty \Psi(t\sqrt{L}) \, F(\cdot,t) \, \frac{dt}{t}.
\]
It is known that the improper integral converges weakly in $L^2(\Ri^d)$ and
\[
\|\pi_{\Psi,L}(F)\|_{L^2(\Ri^d)} \leq C_M \, \|F\|_{T^2_2(\Ri^d)}
\]
for each $M \geq 1$ (cf.\ \cite[p.23]{HLMMY}).

\begin{lemm} \label{atom transform}
Let $B$ be a ball in $\Ri^d$ and $F$ a $T^p_2(\Ri^d)$-atom associated with $B$.
Let $M \geq 1$.
Then there exists a $C_M > 0$ such that $C_M^{-1} \pi_{\Psi,L}(F)$ is a $(p,2,M)$-atom associated with $2B$.
\end{lemm}

\begin{proof}
By definition, we have
\begin{equation} \label{T-atom bound}
\int_{\Ri^d \times (0,\infty)} |F(x,t)|^2 \, dx \, \frac{dt}{t} \leq |B|^{1-2/p}.
\end{equation}
Let $a := \pi_{\Psi,L}(F) = L^M b$, where
\[
b := \int_0^\infty t^{2M} \, t^2 \, L\Phi(-t\sqrt{L}) \, \big( F(\cdot,t) \big) \, \frac{dt}{t}.
\]
Observe that the functions $L^k b$ are all supported on $2B$ for $k = 0, 1, \ldots, M$ as $F$ is supported in $T(B)$.

Next let $g \in L^2(2B)$ and $k \in \{0, 1, \ldots, M\}$.
Let $r$ be the radius of $B$.
Then
\begin{eqnarray*}
\left| \int_{\Ri^d} (r^2 L)^k b g \right|
&=& \left| \lim_{\delta \to 0} \int_{\Ri^d} \left( \int_\delta^{1/\delta} t^{2M} r^{2k} L^k t^2 L\Phi(-t\sqrt{L}) \, \big( F(\cdot, t) \big)(x) \frac{dt}{t} \right) g(x) \, dx \right|
\\
&=& \left| \int_{T(B)} F(x,t) t^{2M} r^{2k} L^k t^2 L\Phi(-t\sqrt{L}) \, g(x) \, dx \, \frac{dt}{t} \right|
\\
&\leq& r^{2M} \, \left( \int_{\Ri^d\times(0,\infty)} \big| F(x,t) \big|^2 \, dx \, \frac{dt}{t} \right)^{1/2} \left( \int_{T(B)} \big| (t^2 L)^{k+1} \Phi(-t\sqrt{L}) \, g(x) \big|^2 \, dx \, \frac{dt}{t} \right)^{1/2}
\\
&\leq& r^{2M} \, |B|^{1/2-1/p} \, \|g\|_{L^2(2B)}.
\end{eqnarray*} 
where the last step follows from \eqref{T-atom bound}, Lemma \ref{aux lemm} and the fact that $k \leq M$.
Consequently
\[
\|(r^2 L)^k b\|_{L^2(2B)} \leq C r^2 \, |B|^{1/2-1/p}
\]
for all $k = 0, 1, \ldots, M$.
The claim now follows.
\end{proof}

\begin{lemm} \label{HpA in Hat}
Let $M \geq 1$.
Then
\[
\left( H_{L,at,M}^p(\Ri^d), \|\cdot\|_{H_{L,at,M}^1(\Ri^d)} \right)
\supset \left( H_A^p(\Ri^d) \cap L^2(\Ri^d), \|\cdot\|_{H_A^p(\Ri^d)} \right).
\]
\end{lemm}

\begin{proof}
Let $f \in H_A^p(\Ri^d)$.
Set $F(\cdot,t) = t\sqrt{L}e^{-t\sqrt{L}}f$.
Then $F \in T^2_2(\Ri^d) \cap T^p_2(\Ri^d)$ by Lemma \ref{aux lemm} and the definition of $H_A^p(\Ri^d)$.
It follows from \cite[Theorem 1.1]{Rus} that $F = \sum_{j \in \Ni} \lambda_j A_j$, where $A_j$'s are $T^p_2(\Ri^d)$-atoms, $\{\lambda_j\} \in l^p$ and 
\begin{equation} \label{lamda norm}
\Big( \sum_{j \in \Ni} |\lambda_j|^p \Big)^{1/p} \leq C \, \|F\|_{T^p_2(\Ri^d)} = C \, \|f\|_{H_A^p(\Ri^d)}.
\end{equation}
Using the Calderon reproducing formula, we obtain
\begin{equation} \label{atom decomp via tent}
f 
= c_\Psi \int_0^\infty \Psi(t\sqrt{L}) \, \big( t\sqrt{L} e^{-t\sqrt{L}} f \big) \, \frac{dt}{t}
= c_\Psi \, \pi_{\Psi,L}(F)
= c_\Psi \, \sum_{j \in \Ni} \lambda_j \, \pi_{\Psi,L}(A_j)
\end{equation}
in $L_2(\Ri^d)$.
Note that $a_j := \pi_{\Psi,L}(A_j)$ is a $(p,2,M)$-atom for $M \geq 1$ and $j \in \Ni$ by Lemma \ref{atom transform}.
Therefore \eqref{atom decomp via tent} is an atomic $(p,2,M)$-representation and hence $f \in H_{L,at,M}^p(\Ri^d)$.
Moreover, 
\[
\|f\|_{H_{L,at,M}^p(\Ri^d)} = \inf \left( \sum_{j \in \Ni} |\lambda_j|^p \right)^{1/p}  \leq C \, \|f\|_{H_A^p(\Ri^d)}
\]
by \eqref{lamda norm}.
\end{proof}

\begin{prop} \label{Hat = HpA}
Let $M > \frac{d}{2}(\frac{1}{p} - \frac{1}{2}) \vee 1$.
We have
\[
\left( H^p_A(\Ri^d) \cap L^2(\Ri^d), \|\cdot\|_{H^p_A(\Ri^d)} \right) 
= \left( H_{L,at,M}^p(\Ri^d), \|\cdot\|_{H_{L,at,M}^1(\Ri^d)} \right).
\]
\end{prop}

\begin{proof}
This follows from Lemmas \ref{Hat in HpA} and \ref{HpA in Hat}.
\end{proof}

\begin{prop} \label{Hat = HpL}
Let $M \geq \frac{d}{2}(\frac{1}{p}-1)$.
We have
\[
\left( H^p_L(\Ri^d), \|\cdot\|_{H^p_L(\Ri^d)} \right) 
= \left( H_{L,at,M}^p(\Ri^d), \|\cdot\|_{H_{L,at,M}^1(\Ri^d)} \right).
\]
\end{prop}

\begin{proof}
This follows from \cite[Theorem 1.4]{SY} and \cite[Theorem 2.15]{BDL}.
\end{proof}

\begin{prop} \label{HpL = HpA}
We have
\[
\left( H^p_L(\Ri^d), \|\cdot\|_{H^p_L(\Ri^d)} \right) 
= \left( H_A^p(\Ri^d) \cap L^2(\Ri^d), \|\cdot\|_{H_{A}^p(\Ri^d)} \right).
\]
\end{prop}

\begin{proof}
The claim is a consequence of Propositions \ref{Hat = HpA} and \ref{Hat = HpL}.
\end{proof}

\section{$H_L^p = H_S^p$} \label{S}

Let $\alpha > 0$.
For each $x \in \Ri^d$, define 
\[
\Gamma_\alpha^{\varepsilon,R}(x) := \{ (y,t) \in \Ri^d \times (\varepsilon,R): |x-y| < \alpha t \},
\]
where $0 < \varepsilon < R < \infty$.

For each $x \in \Ri^d$, define 
\[
S_\alpha f(x) = S_{\alpha,L} f(x)
:= \left( \int_{\Gamma_\alpha(x)} t^2 \, |\nabla_L P_t f(y)|^2 \frac{dy \, dt}{t^{d+1}} \right)^{1/2}
= \left( \int_{\Gamma_\alpha(x)} t^{1-d} \, |\nabla_L P_t f(y)|^2 \, dy \, dt \right)^{1/2}
\]
and
\[
S_\alpha^{\varepsilon,R} f(x) = S_{\alpha,L}^{\varepsilon,R} f(x) 
:= \left( \int_{\Gamma_\alpha^{\varepsilon,R}(x)} t^2 \, |\nabla_L P_t f(y)|^2 \frac{dy \, dt}{t^{d+1}} \right)^{1/2}
= \left( \int_{\Gamma_\alpha^{\varepsilon,R}(x)} t^{1-d} \, |\nabla_L P_t f(y)|^2 \, dy \, dt \right)^{1/2},
\]
where $\Gamma_\alpha(x)$ is defined by \eqref{Gamma_al} and $\nabla_L = (\D_t, \D_1 + x_1, \ldots, \D_d + x_d)$.

Note that $S_1$ coincides with $S$ defined in \eqref{Sop}.
Let $H_S^p(\Ri^d)$ be the completion of
\[
\big\{f \in L^2(\Ri^d): Sf \in L^p(\Ri^d) \big\}
\]
under the quasi-norm
\[
\|\cdot\|_{H_S^p(\Ri^d)} = \|S\cdot\|_{L^p(\Ri^d)}.
\]

\begin{lemm}[{\cite[Lemma 3]{Jiz}}] \label{S<f*}
Let $0 \leq \alpha <1$.
Then there is a $C_\alpha$ such that
\[
S_\alpha^{\varepsilon,R} f(x) \leq C_\alpha \, \big( 1 + |\ln(R/\varepsilon)| \big)^{1/2} \, f_L^*(x)
\]
for all $f \in L_2(\Ri^d)$.
\end{lemm}

Next we define
\[
\widetilde{S}_\alpha^{\varepsilon,R} f(x) 
:= \left( \int_1^2 \int_{\Gamma_{\alpha/a}^{a\varepsilon,aR}(x)} t^{1-d} \, |\nabla_L P_t f(y)|^2 \, dy \, dt \, da \right)^{1/2},
\]
Simple estimation gives
\begin{equation} \label{S equiv}
S_{\alpha/2}^{2\varepsilon, R} f(x)
\leq \widetilde{S}_\alpha^{\varepsilon,R} f(x) 
\leq S_{2\alpha}^{\varepsilon, 2R} f(x).
\end{equation}

\begin{lemm} \label{size S}
There exists a $C > 0$ such that
\[
\big| \{ x \in \Ri^d: \widetilde{S}_{1/20}^{\varepsilon,R} f(x) > 2\lambda \mbox{ and } f_L^*(x) \leq \gamma \, \lambda \} \big| 
\leq 
C \, \gamma^2 \big| \{ x \in \Ri^d: \widetilde{S}_{1/2}^{\varepsilon,R} f(x) > \lambda \big|
\]
for all $0 < \gamma < 1$, $\lambda > 0$, $0 < \varepsilon < R < \infty$ and $f \in H^p_{L,\max}(\Ri^d) \cap L^2(\Ri^d)$.
\end{lemm}

\begin{proof}
The proof follows verbatim to that of \cite[Lemma 4]{Jiz}.
\end{proof}

\begin{lemm} \label{S equiv 2}
Let $\alpha, \beta > 0$ and $0 < \varepsilon < R < \infty$.
Then
\[
\|S_\alpha^{\varepsilon,R}\|_{L_p(\Ri^d)} \sim \|S_\beta^{\varepsilon,R}\|_{L_p(\Ri^d)},
\]
where the implicit constants are independent of $\varepsilon$, $R$ and $f$.
\end{lemm}

\begin{proof}
The proof follows from that of \cite[Proposition 4]{CMS} with obvious modifications.
\end{proof}

\begin{prop} \label{S<f*}
Let $0 < \varepsilon < R < \infty$.
Let $f \in H^p_{L,\max}(\Ri^d)$ be such that $\widetilde{S}_{1/20}^{\varepsilon,R} f \in L^p(\Ri^d)$.
Then there is a $C > 0$ such that
\begin{equation} \label{S<f* in Lp norm}
\|Sf\|_{L^p(\Ri^d)} \leq C \, \|f_L^*\|_{L^p(\Ri^d)}.
\end{equation}
\end{prop}

\begin{proof}
Since $f \in H^p_{L,\max}(\Ri^d)$, we have $f_L^* \in L^p(\Ri^d)$.
This implies $P_s f \in L^p(\Ri^d)$ for all $s > 0$.
We deduce from the definition of $f_L^*$ that $|P_s f(x)| \leq f_L^*(y)$ for all $s > 0$ and for all $x, y \in \Ri^d$ such that $y \in B(x,s)$.
Therefore
\[
|P_s f(x)| \leq \frac{1}{|B(x,s)|} \int_{B(x,s)} f_L^* \leq \frac{C}{s^d} \int_{\Ri^d} f_L^* < \infty
\]
for some $C > 0$.
Hence $P_s f \in L^\infty(\Ri^d)$ for all $s > 0$.
Interpolation gives $P_s f \in L^2(\Ri^d)$ for all $s > 0$.

In what follows, we denote $f_s = P_s f$ for ease of notation.
By Lemma \ref{size S}, there is a $C > 0$ such that
\[
\big| \{ x \in \Ri^d: \widetilde{S}_{1/20}^{\varepsilon,R} f_s(x) > 2\lambda \mbox{ and } (f_s)_L^*(x) \leq \gamma \, \lambda \} \big| 
\leq 
C \, \gamma^2 \big| \{ x \in \Ri^d: \widetilde{S}_{1/2}^{\varepsilon,R} f_s(x) > \lambda \big|
\]
for all $0 < \gamma < 1$ and $\lambda > 0$.
Multiplying both sides by $\lambda^{p-1}$ and then integrating with respect to $\lambda$ give
\begin{equation} \label{S120}
\|\widetilde{S}_{1/20}^{\varepsilon,R} f_s\|_{L^p(\Ri^d)} 
\leq C \, \Big( \gamma^{-1} \, \|(f_s)_L^*\|_{L^p(\Ri^d)} + \gamma^2 \, \|\widetilde{S}_{1/2}^{\varepsilon,R} f_s\|_{L^p(\Ri^d)} \Big).
\end{equation}
It follows from \eqref{S equiv} and Lemma \ref{S equiv 2} that there exists a $C > 0$ such that
\[
\|S_1^{\varepsilon,R} f_s\|_{L^p(\Ri^d)} 
\leq C \, \|S_{1/40}^{\varepsilon, R} f_s\|_{L^p(\Ri^d)}
\leq C \, \|\widetilde{S}_{1/20}^{\varepsilon,R} f_s\|_{L^p(\Ri^d)}.
\]
Also notice that $\|u+v\|_{L^p(\Ri^d)} \leq 2^{(1-p)/p}(\|u\|_{L^p(\Ri^d)} + \|u\|_{L^p(\Ri^d)})$ when $0<p<1$, where $u,v \in \|u\|_{L^p(\Ri^d)}$.
Consequently, there exists a $C>0$ such that
\begin{eqnarray}
\|\widetilde{S}_{1/2}^{\varepsilon,R} f_s\|_{L^p(\Ri^d)}
&\leq& C \, \|S_1^{\varepsilon/2,2R} f_s\|_{L^p(\Ri^d)}
\nonumber
\\
&\leq& C \, \big( \|S_1^{\varepsilon/2,\varepsilon} f_s\|_{L^p(\Ri^d)}
		+ \|S_1^{\varepsilon,2R} f_s\|_{L^p(\Ri^d)} + \|S_1^{R,2R} f_s\|_{L^p(\Ri^d)} \big)
\nonumber
\\
&\leq& C \, \big( \|S_1^{\varepsilon, R} f_s\|_{L^p(\Ri^d)} + \|(f_s)_L^*\|_{L^p(\Ri^d)} \big)	
\nonumber
\\
&\leq& C \, \big( \|\widetilde{S}_{1/20}^{\varepsilon/2,2R} f_s\|_{L^p(\Ri^d)} + \|(f_s)_L^*\|_{L^p(\Ri^d)} \big),
\label{S12}
\end{eqnarray}
where we used Lemma \ref{S<f*} in the third step.
Substituting \eqref{S12} into \eqref{S120} and choosing an appropriate value for $\gamma$, we obtain \eqref{S<f* in Lp norm} for $f_s$.
In addition, we also have that
\[
(f_s)_L^*(x) 
= \sup_{|x-y|<t} \big| P_t f_s(y) \big| 
= \sup_{|x-y|<t} \big| P_{t+s} f(y) \big|
\leq \sup_{|x-y|<t+s} \big| P_{t+s} f(y) \big| = f_L^*(x)
\]
for all $x \in \Ri^d$.
Hence 
\[
\|Sf_s\|_{L^p(\Ri^d)} \leq C \, \|f_L^*\|_{L^p(\Ri^d)}.
\]
Finally we use Lebesgue dominated convergence theorem and take limit $s \longrightarrow 0$ to derive
\[
\|Sf\|_{L^p(\Ri^d)} \leq C \, \|f_L^*\|_{L^p(\Ri^d)}.
\]
This completes the proof.
\end{proof}

We are now ready to prove our main theorem.

\begin{proof}[Proof of Theorem \ref{main theorem}]
($\subset$) This follows from Propositions \ref{HpL subset HpLmap} and \ref{S<f*}.
($\supset$) This is a consequence of Proposition \ref{HpL = HpA} and the fact that $\|\cdot\|_{H_A^p(\Ri^d)} \leq \|\cdot\|_{H_S^p(\Ri^d)}$.
\end{proof}

\section{Boundedness of Riesz transforms} \label{Riesz}

In this section we prove Theorem \ref{Riesz bounded}.
We first consider some auxiliary results.

Define $H^p_{L+2}(\Ri^d)$ as the completion of 
\[
\{ f \in L^2(\Ri^d): \mathcal{M}_{L+2} f \in L^p(\Ri^d) \}
\]
under the quasi-norm
\[
\|\cdot\|_{H_{L+2}^p} = \|\mathcal{M}_{L+2}\cdot\|_{L^p}, 
\]
where $\mathcal{M}_{L+2} f(\cdot) = \sup_{t > 0} |e^{-(L+2)t} f(\cdot)|$ for all $f \in L^2(\Ri^d)$.

Recall the two operators $A$ and $S$ associated with $L$ considered in the previous sections.
In this section we will also consider the operator $L+2$.
Clearly the previous results applied to $L+2$.
To make notation clear, we will write $A_L$ and $S_L$ to emphasize $A$ and $S$ are associated with $L$.
Similarly we can also define $A_{L+2}$ and $S_{L+2}$ associated with $L+2$.

In what follows, we let $L^2_c(\Ri^d)$ be the space of functions in $L^2(\Ri^d)$ with compact supports.

\begin{lemm} \label{L2c}
The following inclusions hold:
\[
\left(L^2_c(\Ri^d), \|\cdot\|_{L^2} \right) 
\subset \left(H^p_L(\Ri^d), \|\cdot\|_{H^p_L} \right)
\subset \left( H^p_{L+2}(\Ri^d), \|\cdot\|_{H^p_{L+2}(\Ri^d)} \right).
\]
Moreover, $L^2_c(\Ri^d)$ is dense in both $H^p_L(\Ri^d)$ and $H^p_{L+2}(\Ri^d)$.
\end{lemm}

\begin{proof}
(First inclusion) Let $f \in L^2_c(\Ri^d)$.
Then by a generalised Holder's inequality (cf.\ \cite[Corollary 2.5]{AF})
\[
\|Af\|_{L^p(\Ri^d)} 
\leq |\supp f|^{1/q} \, \|Af\|_{L^2(\Ri^d)}
\leq C \, |\supp f|^{1/q} \, \|f\|_{L^2(\Ri^d)} 
< \infty,
\]
where $1/q = 1/p - 1/2$ and we used Lemma \ref{A L2 bounded} in the second step.
The first inclusion now follows from Proposition \ref{HpL = HpA}.

(Second inclusion) Let $f \in H^p_L(\Ri^d)$.
Then $\mathcal{M}_L f \in L^p(\Ri^d)$.
However 
\[
\mathcal{M}_{L+2} f(x) = \sup_{t > 0} |e^{-2t} \, T_t f(x)| \leq \sup_{t > 0} |T_t f(x)| = \mathcal{M}_L f
\]
for all $x \in \Ri^d$.
From this we deduce that $f \in H^p_{L+2}(\Ri^d)$ and the second inclusion holds.

(Density) By the atomic characterisation in Proposition \ref{atom decom}, each function in $H^p_L(\Ri^d)$ can be approximated by a finite linear combination $H^p_L$-atoms.
But each such finite linear combination is clearly in $L^2_c(\Ri^d)$.
Hence the claim follows.
\end{proof}

The following lemma is immediate from Lemma \ref{L2c}.

\begin{lemm} \label{increment by 2}
We have
\[
\left( H^p_L(\Ri^d), \|\cdot\|_{H^p_L(\Ri^d)} \right) 
= \left( H^p_{L+2}(\Ri^d), \|\cdot\|_{H^p_{L+2}(\Ri^d)} \right).
\]
\end{lemm}

\begin{proof}[Proof of Theorem \ref{Riesz bounded}]
Let $j \in \{1, 2, 3, \ldots\}$.
We will show that $R_j^L$ is bounded on $H^p_{L}(\Ri^d) \cap L^2(\Ri^d)$.
The claim then follows by density of $H^p_{L}(\Ri^d) \cap L^2(\Ri^d)$ in $H^p_{L}(\Ri^d)$.

Let $f \in H^p_{L}(\Ri^d) \cap L^2(\Ri^d)$.
Then $f \in H^p_{L+2}(\Ri^d) \cap L^2(\Ri^d)$ by Lemma \ref{increment by 2}.
There exists a $C > 0$ such that
\begin{eqnarray*}
\|R_j^L f\|_{H^p_{L+2}(\Ri^d)} 
& \leq & C \, \|A_{L+2} R_j^L f\|_{L^p(\Ri^d)}
\\
&=& C \, \left\|\left( \int_0^\infty \int_{|x-y|<t} t^{1-d} \, \left|(\D_t e^{-t(L+2)^{1/2}} R_j^L f)(y) \right|^2 \, dy \, dt \right)^{1/2} \right\|_{L^p(\Ri^d)}
\\
&=& C \, \left\|\left( \int_0^\infty \int_{|x-y|<t} t^{1-d} \, \left| \Big(t (\D_j + x_j) e^{-tL^{1/2}} f \Big)(y) \right|^2 \, dy \, dt \right)^{1/2} \right\|_{L^p(\Ri^d)}
\\
&\leq& C \, \|S_L f\|_{L^p(\Ri^d)}
\leq C \, \|f\|_{L^p(\Ri^d)},
\end{eqnarray*}
where we used Proposition \ref{HpL = HpA} in the first step, \cite[Lemma 8]{Jiz} in the third step and Theorem \ref{main theorem} in the last step.
\end{proof}

%\begin{proof}[Proof of Corollary \ref{Riesz char}]
%($\subset$) Let $f \in H^p_L(\Ri^d)$. 
%Then $f \in L^p(\Ri^d)$ and $R_j^L f \in L^p(\Ri^d)$ by Theorem \ref{Riesz bounded}.
%($\supset$) Let $f \in L^p(\Ri^d)$ be such that $R_j^L f \in L^p(\Ri^d)$.
%Then ????????
%\end{proof}

%...

%\subsection*{Acknowledgements}
%We thank The Anh Bui for fruitful discussions during the preparation of the manuscript.

%\bibliographystyle{Bibstyle}
%\bibliography{Bib}

\newcommand{\etalchar}[1]{$^{#1}$}

\end{document}